\documentclass[graybox]{Style/svmult}

\usepackage{type1cm}        
\usepackage{makeidx}         
\usepackage{graphicx}        
\usepackage{multicol}        
\usepackage[bottom]{footmisc}

\usepackage{newtxtext}       %
\usepackage{newtxmath}       

\usepackage{hyperref}
\usepackage{scrextend}
\usepackage{float}

\makeindex             

\newcommand{\cpm}[1]{\tilde{#1}}
\begin{document}
	
	\title*{A Convergence Analysis of the Parallel Schwarz Solution of the Continuous Closest Point Method}
	\titlerunning{Convergence of PS-CPM}
	
	\author{Alireza Yazdani, Ronald D. Haynes, and Steven J. Ruuth}
	
	\institute{Alireza Yazdani, Steven J. Ruuth \at Simon Fraser University, 8888 University Dr., Burnaby, BC V5A 1S6, \email{alirezay@sfu.ca,sruuth@sfu.ca}
		\and Ronald D. Haynes \at Memorial University of Newfoundland, 230 Elizabeth Ave, St. Johns, NL A1C 5S7, \email{rhaynes@mun.ca}}
	
	\maketitle
	\abstract*{The discretization of surface intrinsic PDEs has challenges that one might not face in the flat space. The closest point method (CPM) is an embedding method that represents surfaces using a function that maps points in the flat space to their closest points on the surface. This mapping brings intrinsic data onto the embedding space, allowing us to numerically approximate PDEs by the standard methods in the tubular neighborhood of the surface. Here, we solve the surface intrinsic positive Helmholtz equation by the CPM paired with finite differences which usually yields a large, sparse, and non-symmetric system. Domain decomposition methods, especially Schwarz methods, are robust algorithms to solve these linear systems. While there have been substantial works on Schwarz methods, Schwarz methods for solving surface differential equations have not been widely analyzed. In this work, we investigate the convergence of the CPM coupled with Schwarz method on 1-manifolds in d-dimensional space of real numbers.}
	\section{Introduction}
	\label{sec:yazdani_contrib_intro}
	
	Consider the surface intrinsic positive Helmholtz equation
	\begin{equation}\label{eq:yazdani_contrib_helmholtz}
	(c - \Delta_\mathcal{S}) u = f,
	\end{equation}
	where $\Delta_\mathcal{S}$ denotes the Laplace-Beltrami operator associated with the surface $\mathcal{S} \subset \mathbb{R}^d$, and $c > 0$ is a constant. Discretization of this equation arises in many applications including  the time-stepping of reaction-diffusion equations on surfaces \cite{Macdonald9209}, the comparison of shapes \cite{REUTER2009739}, and the solution of Laplace-Beltrami eigenvalue problems \cite{Macdonald_2011}. As a consequence,
	considerable recent work has taken place to develop efficient, high-speed solvers for this and other related PDEs on surfaces.
	
	There are several methods to solve surface intrinsic differential equations (DEs).
	If a surface parameterization (a mapping from the surface to a parameter space) is known, then the equation can be solved in the parameter domain \cite{degener2003}.
	For triangulated surfaces, a finite element discretization can be created \cite{FED}.
	Alternatively, we can solve the DE in a neighborhood of the surface using standard PDE methods in the underlying embedding space \cite{BERTALMIO2001759,CPM2008,Chu2018,Martin2020EquivalentEO}. 
	Here, we discretize via the {\it closest point method} (CPM), which is an embedding method suitable for the discretization of PDEs on surfaces. The closest point method leads to non-symmetric linear systems to solve. On complex geometries or when varying scales arise, iterative solvers can be slow despite the sparsity of the underlying systems. In order to develop an efficient iterative solver which is also capable of parallelism, Parallel Schwarz (PS) and Optimized Parallel Schwarz (OPS) algorithms have been applied to the CPM for \eqref{eq:yazdani_contrib_helmholtz} in \cite{May2020DD}. Here, we study the convergence of the PS-CPM at the continuous level for smooth, closed 1-manifolds where periodicity is inherent in the geometry.  As shown in Section~\ref{sec:yazdani_contrib_RAS-CPM}, this problem, posed in $\mathbb{R}^d$, is  equivalent to a one-dimensional periodic problem.  This leads us to study the 1-dimensional periodic problem in detail.
	
	While there has been substantial work carried out on Schwarz methods, they have not been widely used for solving surface DEs. The shallow-water equation is solved with a PS iteration on the cubed-sphere with a finite volume discretization in \cite{Yang2010}. PS and OPS methods for the 2D positive definite Helmholtz problem are investigated on the unit sphere in \cite{LOISEL2010}. In \cite{LOISEL2010}, the analysis is based on latitudinal subdomains that are periodic in longitude. Hence, the Fourier transform is a natural choice to solve the subproblems analytically and obtain the contraction factor. PS and OPS methods are also analyzed with an overset grid for the shallow-water equation in \cite{QADDOURI2008}. In that work, the discretization in 1D is reduced to the positive definite Helmholtz problem on the unit circle. The unit circle case is investigated with two equal-sized subdomains, and a convergence factor is derived for the configuration in terms of the overlap parameter.  In addition, the 2D positive definite Helmholtz problem on the sphere is analyzed where the subdomains are derived from a Yin-Yang grid system.  It is worth noting a key difference between our work and \cite{QADDOURI2008}.   In our problem, domain subdivision is carried out in the underlying embedding space.  As a consequence, the unequal-sized subdomain case is essential to our understanding of the problem. 
	
	The convergence of PS and OPS for general surfaces remains unknown. Section~\ref{sec:yazdani_contrib_CPM} reviews the CPM. Section~\ref{sec:yazdani_contrib_RAS-CPM} studies the PS-CPM combination for the surface intrinsic positive Helmholtz equation \eqref{eq:yazdani_contrib_helmholtz} by analyzing an equivalent one-dimensional periodic problem. This section proves convergence and derives convergence factors. 
	Although \eqref{eq:yazdani_contrib_helmholtz} on 1-manifolds can be solved through parameterization, we only investigate the convergence of the PS-CPM for 1-manifolds in this paper with the hope of extending our work to higher dimensional manifolds in the future.
	Section~\ref{sec:yazdani_contrib_NumSim} provides a numerical experiment in which the PS-CPM contraction factor converges to its PS counterpart by increasing the grid resolution. Finally, Section~\ref{sec:yazdani_contrib_Conc} gives conclusions.
	\section{The Closest Point Method}
	\label{sec:yazdani_contrib_CPM}
	
	The CPM was first introduced in \cite{CPM2008} for explicitly solving evolutionary PDEs on surfaces. 
	It is an embedding method and allows the use of standard Cartesian methods for the discretization of surface intrinsic differential operators. The surface representation and extension of quantities defined on the surface to the surrounding embedding space  is done using the closest point mapping
	$\text{cp}_\mathcal{S}(x) = \underset{s \in \mathcal{S}}{\text{arg\,min}} |x - s|$
	for $x \in \mathbb{R}^d$.  This mapping gives the closest point in Euclidean distance to the surface for any point $x$ in the embedding space. It is smooth for any point in the embedding space within a distance $R_0$ of a smooth surface, where $R_0$ is a lower bound for the surface radii of curvature \cite{Chu2018}.
	
	Suppose the closest point mapping of a manifold is smooth over a tubular neighborhood $\Omega \subset \mathbb{R}^d$ of the manifold. We introduce $\cpm{u}:\Omega \rightarrow \mathbb{R}$ as the solution to the embedding CPM problem.
	Two principles are fundamental to the CPM: \textit{equivalence of gradients} and \textit{equivalence of divergence} \cite{CPM2008}.
	Assuming a smooth manifold $\mathcal{S}$, the equivalence of gradients principle gives us $\nabla \cpm{u}(\text{cp}_\mathcal{S}) = \nabla_\mathcal{S}u$ since $\cpm{u}(\text{cp}_\mathcal{S})$ is constant in the normal direction to the manifold.
	Further, applying the equivalence of divergence principle, $\nabla \cdot (\nabla \cpm{u}(\text{cp}_\mathcal{S})) = \nabla_\mathcal{S} \cdot (\nabla \cpm{u}(\text{cp}_\mathcal{S}))$ holds on the manifold.  Therefore, on the manifold, 
	\begin{equation} \label{eq:yazdani_contrib_equivlap}
	\Delta \cpm{u}(\text{cp}_\mathcal{S}) = \nabla \cdot (\nabla \cpm{u}(\text{cp}_\mathcal{S})) = \nabla_\mathcal{S} \cdot (\nabla \cpm{u}(\text{cp}_\mathcal{S})) = \nabla_\mathcal{S} \cdot (\nabla_\mathcal{S} u) = \Delta_\mathcal{S} u.
	\end{equation}
	A modified version of \eqref{eq:yazdani_contrib_equivlap} offers improved stability at the discrete level and is normally used in elliptic problems \cite{ImplicitCPM2008,Macdonald_2011,vonglehn2013embedded}. 
	The regularized Laplace operator is
	\begin{equation}\label{eq:yazdani_contrib_regCPM}
	\Delta^{\#}_h \cpm{u} = \Delta \cpm{u}(\text{cp}_\mathcal{S}) - \frac{2d}{h^2}\left[ \cpm{u} - \cpm{u}(\text{cp}_\mathcal{S}) \right],
	\end{equation}
	where $0 < h \ll 1$.  As in \cite{ImplicitCPM2008,Macdonald_2011}, we take the parameter $h$ to be equal to the mesh spacing in the fully discrete setting.
	
	Equation~\eqref{eq:yazdani_contrib_regCPM} gives our replacement for the Laplace-Beltrami operator. 
	Applying it, and extending the function $f$ off the manifold using the closest point mapping gives our embedding equation for \eqref{eq:yazdani_contrib_helmholtz}: 
	\begin{equation} \label{eq:yazdani_contrib_CPM}
	(c - \Delta^{\#}_h) ~\cpm{u}= f(\text{cp}_\mathcal{S}), ~~~x\in \Omega.
	\end{equation}
	
	Standard numerical methods in the embedding space may be applied to \eqref{eq:yazdani_contrib_CPM} to complete the discretization.   In this paper, we apply standard second order finite differences on regular grids to approximate the derivative operators. Because discrete points do not necessarily lie on $\mathcal{S}$, an interpolation scheme is needed to recover surface values. Utilizing tensor product barycentric Lagrangian interpolation \cite{berrut2004barycentric}, an extension matrix $\textbf{E}$ is defined to extend values off of the manifold. Note that the extension matrix may be viewed as a discretization of the closest point mapping.
	
	Using a mesh spacing $h$ and degree-$p$ interpolation polynomials, it is sufficient to numerically approximate equation \eqref{eq:yazdani_contrib_CPM} in a narrow tube around $\mathcal{S}$ of radius $r = \sqrt{(d-1)(p+1)^2+(p+3)^2} h / 2$. A more thorough explanation of the CPM at the discrete level can be found in \cite{ImplicitCPM2008}.
	\section{The PS-CPM Convergence Analysis}
	\label{sec:yazdani_contrib_RAS-CPM}
	
	PS is an overlapping domain decomposition method which is designed to iteratively solve DEs over subdomains, distributing the computational costs.
	It is also capable of parallelism and can be combined with the CPM, a method whose underlying linear system is sparse.
	We assume $\mathcal{S}$ to be a smooth, closed 1-manifold in $\mathbb{R}^d$ with arclength $L$. 
	We consider the case with two subdomains, but the discussion can be generalized to any finite number of subdomains \cite{Alireza2021}. 
	We let the disjoint subdomains be $\tilde{\mathcal{S}_1}$ and $\tilde{\mathcal{S}_2}$.
	We parameterize the manifold by arclength $s$ starting at a boundary of $\tilde{\mathcal{S}_1} $. 
	Next, we let the overlapping subdomains be $\mathcal{S}_1 = [a_1, b_1]$ and $\mathcal{S}_2 = [a_2, b_2]$.
	Since overlapping subdomains are needed, we have $a_1 < 0$ and $b_2 > L$.  Define $\ell_1 \equiv b_1 - a_1$ and $\ell_2 \equiv b_2 - a_2$ to be the subdomain lengths. 
	Further,  let $\delta_1 = b_1 - a_2$ and $\delta_2 = b_2 - (a_1 + L)$ denote the subdomain overlaps at $s=\ell_1$ and $s=\ell_2$, respectively.
	In addition, we assume $0 < \delta_1 + \delta_2 < \min\{\ell_1, \ell_2\}$. 
	In the CPM, the overlapping subdomains  $\Omega_1$ and $\Omega_2$, corresponding to $\mathcal{S}_1$ and $\mathcal{S}_2$, are constructed using a graph-based partitioning algorithm applied over the computational tube \cite{May2020DD}.
	Then, the PS-CPM for equation \eqref{eq:yazdani_contrib_helmholtz} is: for $n=0,1,\ldots$ and for $j=1,2$ solve 
	\begin{equation} \label{eq:yazdani_contrib_RASCPM}
	\begin{cases}
	(c - \Delta^{\#}_h)~\cpm{u}^{n+1}_j = f(\text{cp}_\mathcal{S}), & \text{in } \Omega_j, \\
	\cpm{u}^{n+1}_j = \cpm{u}^n(\text{cp}_\mathcal{S}), & \text{on } \Gamma_{jk}, k \neq j 
	\end{cases}
	\end{equation}
	where $\Gamma_{jk}$ for $j, k = 1, 2$ are the boundaries of subdomains $j$ and $k$.
	
	To begin, an initial guess is needed over the subdomain boundaries. 
	An iteration may then be completed by solving all subproblems. 
	This gives new boundary values that can be used to initiate the next iteration, and so on, until convergence.
	In this form of the Schwarz algorithm, there is no concept of a global solution.
	In order to construct the global solution, a weighted average of subdomain solutions is utilized \cite{Efstathiou2003}.
	In this paper, at any time, the approximation of the global solution is given as the union of the disjoint subdomain solutions $u^n = u^n_1|_{\tilde{\mathcal{S}}_1} \cup u^n_2|_{\tilde{\mathcal{S}}_2}$.
	This is called restricted additive Schwarz (RAS), and we use the labels PS and RAS interchangeably.
	Our analysis examines the equivalent one dimensional periodic problem formulated below.
	
	\begin{theorem}\label{thm:yazdani_contrib_theorem1}
		In the limit as $h\rightarrow 0$,  and using two subdomains $\mathcal{S}_1 = [a_1, b_1]$ and $\mathcal{S}_2 = [a_2, b_2]$, the PS-CPM for the positive surface intrinsic Helmholtz equation \eqref{eq:yazdani_contrib_RASCPM} is equivalent to:
		\begin{equation} \label{eq:yazdani_contrib_RASCPM2}
		\begin{cases}
		(c - \frac{\mathrm{d}^2}{\mathrm{d}s^2})u^{n+1}_1 = f, & \text{in } \mathcal{S}_1, \\
		u^{n+1}_1(a_1) = u^n_2(a_1+L), \\
		u^{n+1}_1(b_1) = u^n_2(b_1),
		\end{cases}
		\quad , \quad
		\begin{cases}
		(c - \frac{\mathrm{d}^2}{\mathrm{d}s^2})u^{n+1}_2 = f, & \text{in } \mathcal{S}_2, \\
		u^{n+1}_2(a_2) = u^n_1(a_2), \\
		u^{n+1}_2(b_2) = u^n_1(b_2-L),
		\end{cases}
		\end{equation}
		where $L$ is the manifold length.
	\end{theorem}
	\begin{proof}
		
		For a smooth manifold $\mathcal{S}$, the regularized operator $\Delta^{\#}_h$ is consistent with the Laplace operator on the manifold \cite{ImplicitCPM2008}. Thus the CPM  is consistent with the surface intrinsic PDE problems in the limit $h\rightarrow 0$ where $h$ denotes the mesh size.
		Parameterizing a one-dimensional manifold $\mathcal{S}$ in $\mathbb{R}^d$ by arclength $s$, the differential operator $\Delta_\mathcal{S}$ becomes $\mathrm{d}^2/\mathrm{d}s^2$, yielding our result.
	\end{proof}
	
	In \cite{QADDOURI2008}, the convergence of \eqref{eq:yazdani_contrib_RASCPM2} is studied for an equal-sized partitioning. The partitioning arising from the PS-CPM problems in \eqref{eq:yazdani_contrib_RASCPM} is performed within the embedding space.  As a consequence, our subdomains will be unequal.
	This motivates us to investigate the convergence of the method for an unequal-sized partitioning.
	
	By defining the errors $\epsilon_j^n = u_j^n - u|_{\mathcal{S}_j}$,  $j = 1, 2$, and using the linearity of \eqref{eq:yazdani_contrib_helmholtz}, iteration \eqref{eq:yazdani_contrib_RASCPM2} is reduced to:
	\begin{equation} \label{eq:yazdani_contrib_error-RASCPM}
	\begin{cases}
	(c - \frac{\mathrm{d}^2}{\mathrm{d}s^2})\epsilon^{n+1}_1 = 0, & \text{in } \mathcal{S}_1, \\
	\epsilon^{n+1}_1(a_1) = \epsilon^n_2(a_1+L), \\
	\epsilon^{n+1}_1(b_1) = \epsilon^n_2(b_1),
	\end{cases}
	\quad , \quad
	\begin{cases}
	(c - \frac{\mathrm{d}^2}{\mathrm{d}s^2})\epsilon^{n+1}_2 = 0, & \text{in } \mathcal{S}_2, \\
	\epsilon^{n+1}_2(a_2) = \epsilon^n_1(a_2), \\
	\epsilon^{n+1}_2(b_2) = \epsilon^n_1(b_2-L).
	\end{cases}
	\end{equation}
	
	After solving the ODEs in \eqref{eq:yazdani_contrib_error-RASCPM}, error values at the boundaries can be computed. At each iteration, these error values depend on the error values at the boundaries from the previous iteration.    To state this concisely, we define an error vector at iteration $n$ which is comprised of the error values at the boundaries:
	\begin{equation}\label{eq:yazdani_contrib_errorvec}
	\boldsymbol{\mathrm{\epsilon}}^{n} := [\epsilon_1^{n}(b_2-L), \epsilon_1^{n}(a_2), \epsilon_2^{n}(b_1), \epsilon_2^{n}(a_1+L)]^T.
	\end{equation}
	We obtain, in matrix form,
	$\boldsymbol{\mathrm{\epsilon}}^{n+1} = \boldsymbol{\mathrm{M}_\text{PS}}\boldsymbol{\mathrm{\epsilon}}^{n},$
	where
	\begin{equation}\label{eq:yazdani_contrib_RASMat}
	\boldsymbol{\mathrm{M}}_\text{PS} = 
	\begin{bmatrix}
	0 & 0 & r_1 & p_1 \\
	0 & 0 & q_1 & s_1 \\
	r_2 & p_2 & 0 & 0 \\
	q_2 & s_2 & 0 & 0 \\
	\end{bmatrix}
	\end{equation}
	is called the iteration matrix.  It has entries 
	\begin{equation*}\textit{}
		p_j = \frac{1-e^{2\sqrt{c}(\ell_j-\delta_{j-1})}}{1-e^{2\sqrt{c}\ell_j}}e^{\sqrt{c}\delta_{j-1}}, \qquad
		r_j = \frac{1-e^{2\sqrt{c}\delta_{j-1}}}{1-e^{2\sqrt{c}\ell_j}}e^{\sqrt{c}(\ell_j-\delta_{j-1})},
	\end{equation*}
	\begin{equation} \label{eq:yazdani_contrib_quant}
	q_j = \frac{1-e^{2\sqrt{c}(\ell_j-\delta_j)}}{1-e^{2\sqrt{c}\ell_j}}e^{\sqrt{c}\delta_j}, \qquad s_j = \frac{1-e^{2\sqrt{c}\delta_j}}{1-e^{2\sqrt{c}\ell_j}}e^{\sqrt{c}(\ell_j-\delta_j)},
	\end{equation}
	for $j = 1, 2$ and $\delta_0 \equiv \delta_2$. 
	The definitions of  $\delta_j$ and $\ell_j$ may be found at the beginning of this section.
	The following lemma holds for the quantities in \eqref{eq:yazdani_contrib_quant}:
	
	\begin{lemma}[\cite{Alireza2021}]\label{lemma:yazdani_contrib_quant}
		Suppose $0 < \delta_1 + \delta_2 < \min\{\ell_1, \ell_2\}$.   Then the scalars $p_j, q_j, r_j, s_j, j=1,2,$ appearing in \eqref{eq:yazdani_contrib_quant} satisfy $0 < q_j + s_j < 1$ and $0 < p_j + r_j < 1$.
	\end{lemma}
	Now, we arrive at the most important result of this section.
	\begin{theorem}
		Under the restrictions on the partitioning of the manifold $\mathcal{S}$ detailed in Lemma~1 above, the PS iteration \eqref{eq:yazdani_contrib_RASCPM2} for the positive Helmholtz equation on any closed, smooth one-dimensional manifold converges globally.
	\end{theorem}
	\begin{proof}
		We must show the spectral radius of the iteration matrix, $\rho(\boldsymbol{\mathrm{M}}_\text{PS})$,  is less than 1. $\|\boldsymbol{\mathrm{M}}_\text{PS}\|_\infty$ bounds the spectral radius,
		$\rho(\boldsymbol{\mathrm{M}}_\text{PS}) \leq \|\boldsymbol{\mathrm{M}}_\text{PS}\|_\infty = \max \{r_j + p_j, q_j + s_j\}.$
		In Lemma \ref{lemma:yazdani_contrib_quant}, we have shown that $0 < p_j + r_j < 1$ and $0 < q_j + s_j < 1$. Therefore, $\|\boldsymbol{\mathrm{M}}_\text{PS}\|_\infty < 1$, and consequently the algorithm converges.
	\end{proof}
	
	We define the convergence factor $\kappa$ as the ratio of the $\infty$-norm of the error vector \eqref{eq:yazdani_contrib_errorvec} at two steps $n+2$ and $n$, 
	$\kappa = \|\boldsymbol{\mathrm{\epsilon}}^{n+2}\|_\infty/\|\boldsymbol{\mathrm{\epsilon}}^{n}\|_\infty$.
	Considering the inequality $\|\boldsymbol{\mathrm{\epsilon}}^{n+1}\|_\infty \leq \|\boldsymbol{\mathrm{M}}_\text{PS}\|_\infty\|\boldsymbol{\mathrm{\epsilon}}^{n}\|_\infty$, $\|\boldsymbol{\mathrm{M}}_\text{PS}\|_\infty^2$ is an upper bound for the convergence factor. That is, $\kappa \leq \|\boldsymbol{\mathrm{M}}_\text{PS}\|_\infty^2$.
	In the following corollary, we show that the our analysis for the equal-sized partitioning agrees with the one obtained in \cite{QADDOURI2008}.
	\begin{corollary}\label{cor:yazdani_contrib_1}
		Assume an equal-sized partitioning for the PS iteration \eqref{eq:yazdani_contrib_RASCPM2}. That is, $\mathcal{S}_1 = [-\delta, L/2+\delta]$, $\mathcal{S}_2 = [L/2-\delta, L+\delta]$. Then, the convergence factor can be calculated as $\kappa \leq (p + r)^2 = (e^{\sqrt{c}L/2}+e^{\sqrt{c}\delta})^2/(1+e^{\sqrt{c}(L/2+\delta)})^2$.
	\end{corollary}
	
	\begin{proof}
		If we make the simplifying assumption that both subdomains are of equal size and have a common overlap size, then $q_1 = q_2 = p_1 = p_2 = p$ and $s_1 = s_2 = r_1 = r_2 = r$. The iteration matrix becomes a doubly stochastic matrix with row and column sums of $p+r$, and subsequently $\rho(\boldsymbol{\mathrm{M}}_\text{PS})=p+r$. By a direct substitution for $p$ and $r$, we obtain $\kappa = \rho(\boldsymbol{\mathrm{M}}_\text{PS})^2 = (e^{\sqrt{c}L/2}+e^{\sqrt{c}\delta})^2/(1+e^{\sqrt{c}(L/2+\delta)})^2$.
	\end{proof}
	\section{Numerical Simulation}
	\label{sec:yazdani_contrib_NumSim}
	
	Here we numerically verify the results obtained in Section \ref{sec:yazdani_contrib_RAS-CPM}. Since numerical solutions of the PS-CPM and the PS algorithm will be compared, we use RAS as the domain decomposition method to build a global approximate solution. It is shown in \cite{Efstathiou2003} that RAS and PS are identical iterations and have the same convergence rate. Hence, we will use RAS-CPM instead of PS-CPM hereafter. 
	
	Theorem \ref{thm:yazdani_contrib_theorem1} shows that the CPM equipped with RAS as a solver is in the limit as $h\rightarrow 0$ equivalent to RAS applied to a 1D periodic problem.
	To verify this, we numerically solve \eqref{eq:yazdani_contrib_helmholtz} with $c=1$ and $f(s) = \sin(2\pi s/L)$ using the RAS-CPM for the boundary of a M\"obius strip with width 1, whose center circle has radius 1. The initial guess for the discrete solution is taken as $U^{(0)} = 0$.  Two disjoint subdomains are created by splitting the length of the curve in a 1:2 ratio, and overlapping subdomains are formed using overlaps $\delta = \delta_1 = \delta_2 = 0.1L$. The solution using the RAS-CPM with grid spacing $h = 0.01$ and fourth degree barycentric Lagrangian interpolation applied in a dimension-by-dimension fashion is shown in Fig.~\ref{fig:yazdani_contrib_RAS-CPM} (left).   Here,  the disjoint subdomains are visualized as point clouds. Convergence histories for various grid spacings are depicted in Fig.~\ref{fig:yazdani_contrib_RAS-CPM} (right).  Here, the RAS and the RAS-CPM contraction factors are compared with the theoretical result. The errors are defined as the max-norm of the difference of the DD solution and the single domain solution.
	As we observe in Fig.\ref{fig:yazdani_contrib_RAS-CPM} (right), the RAS error has the same decay rate as that described in Theorem~\ref{thm:yazdani_contrib_theorem1} (shown as the dashed line).  In addition, the RAS-CPM error tends toward the RAS error as the mesh size is reduced. 
	
	\begin{figure}[hpbt]
		\centering
		\includegraphics[width=0.49\linewidth]{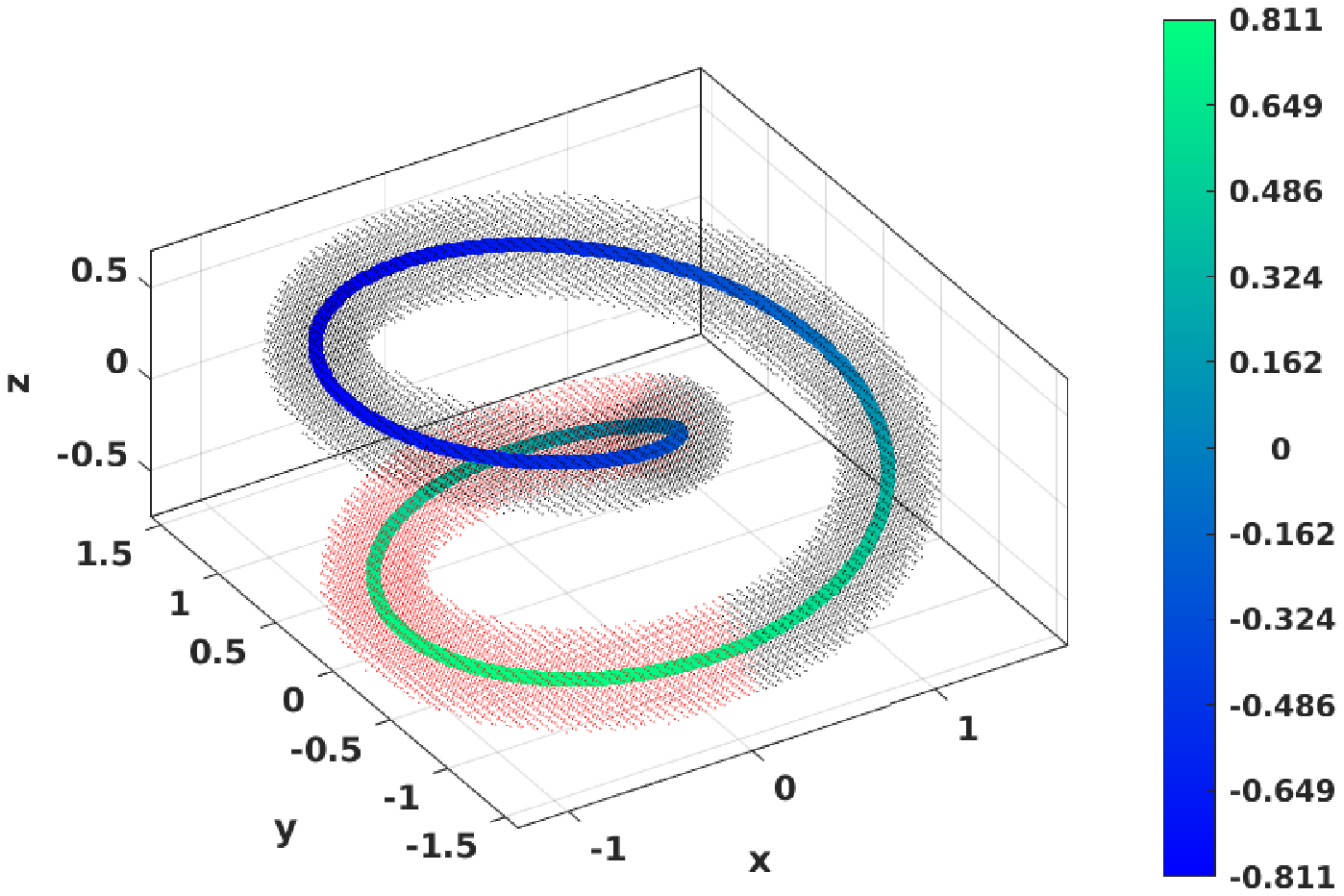}
		\includegraphics[width=0.49\linewidth]{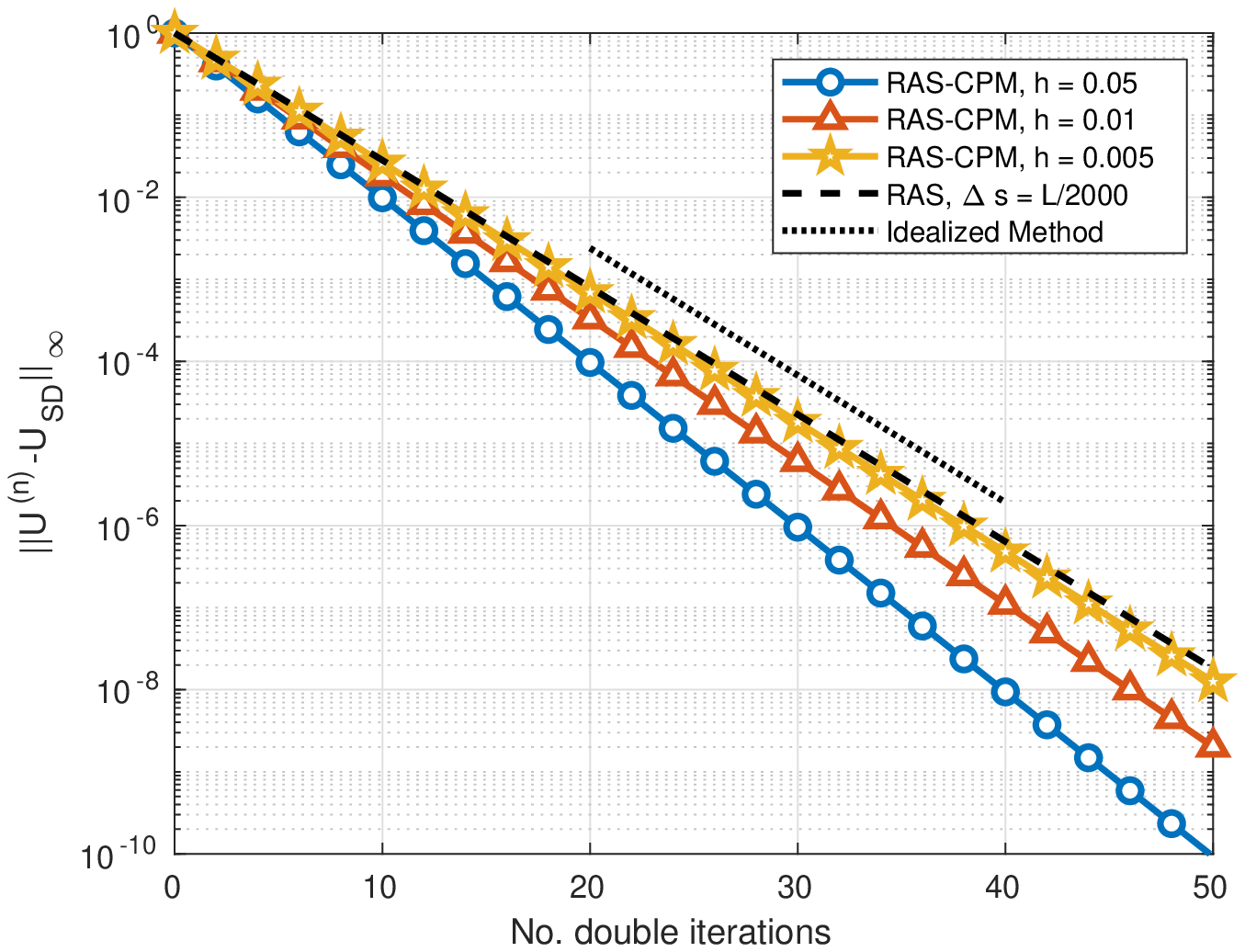}
		\caption{\textbf{Left:} RAS-CPM solution of the surface intrinsic Helmholtz equation on edge of a M\"obius strip. The disjoint subdomains are depicted. \textbf{Right:} Error versus the double iteration number.}
		\label{fig:yazdani_contrib_RAS-CPM}
	\end{figure}
	
	As another experiment, \eqref{eq:yazdani_contrib_helmholtz} is solved with two equal-sized subdomains, assuming $\mathcal{S}$ is the unit circle. The disjoint subdomains are shown in Fig.~\ref{fig:yazdani_contrib_Conv} (left). Fig.~\ref{fig:yazdani_contrib_Conv} (right) shows the effect of the overlap parameter $\delta$ on RAS-CPM for three different grids ($h=0.05, 0.01, 0.005$). 
	For a given $h$ and $\delta$, the numerical convergence factor changes slightly as the iteration progresses, hence we present an average of the convergence factor over all iterations.
	To compare with the result in Corollary \ref{cor:yazdani_contrib_1}, the theoretical convergence factor associated with a double iteration, $(e^{L/2}+e^{\delta})^2/(1+e^{L/2+\delta})^2$, is shown in Fig.~\ref{fig:yazdani_contrib_Conv} (right) as a dashed line. 
	The observed RAS-CPM contraction factor converges to the theoretical value as the grid quality improves.
	By increasing the overlap, $\kappa$ is reduced and a better convergence factor is obtained.
	
	\begin{figure}[t]
		\centering
		\includegraphics[width=0.49\linewidth]{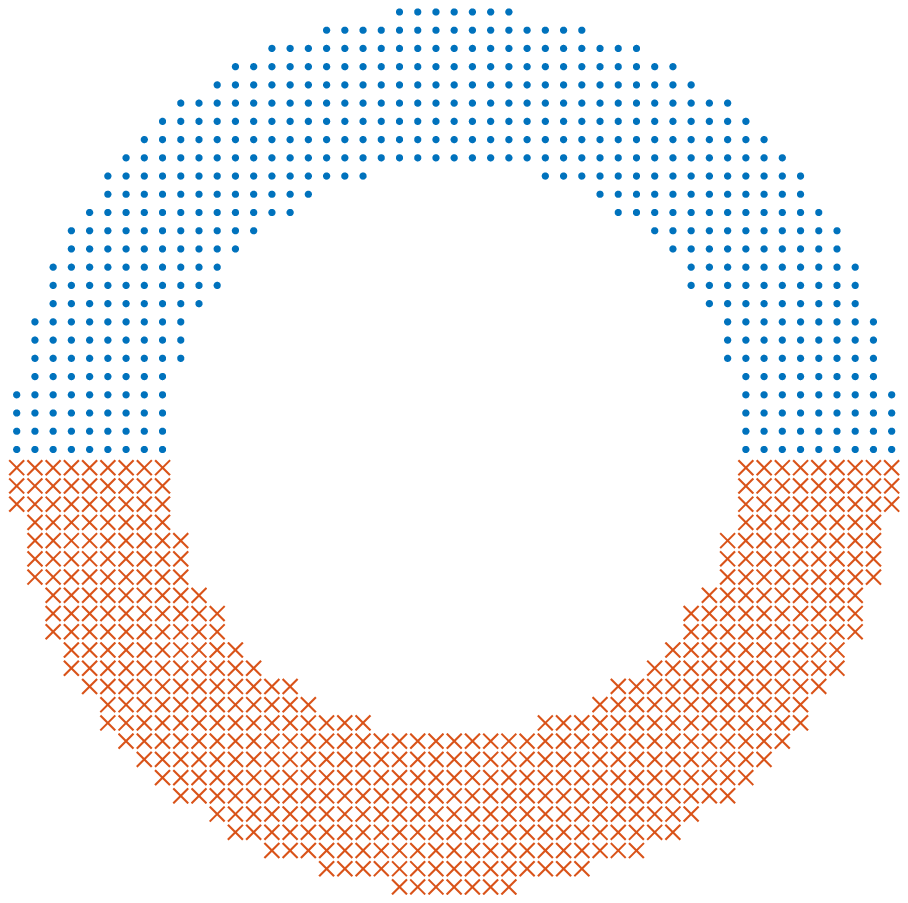}
		\includegraphics[width=0.49\linewidth]{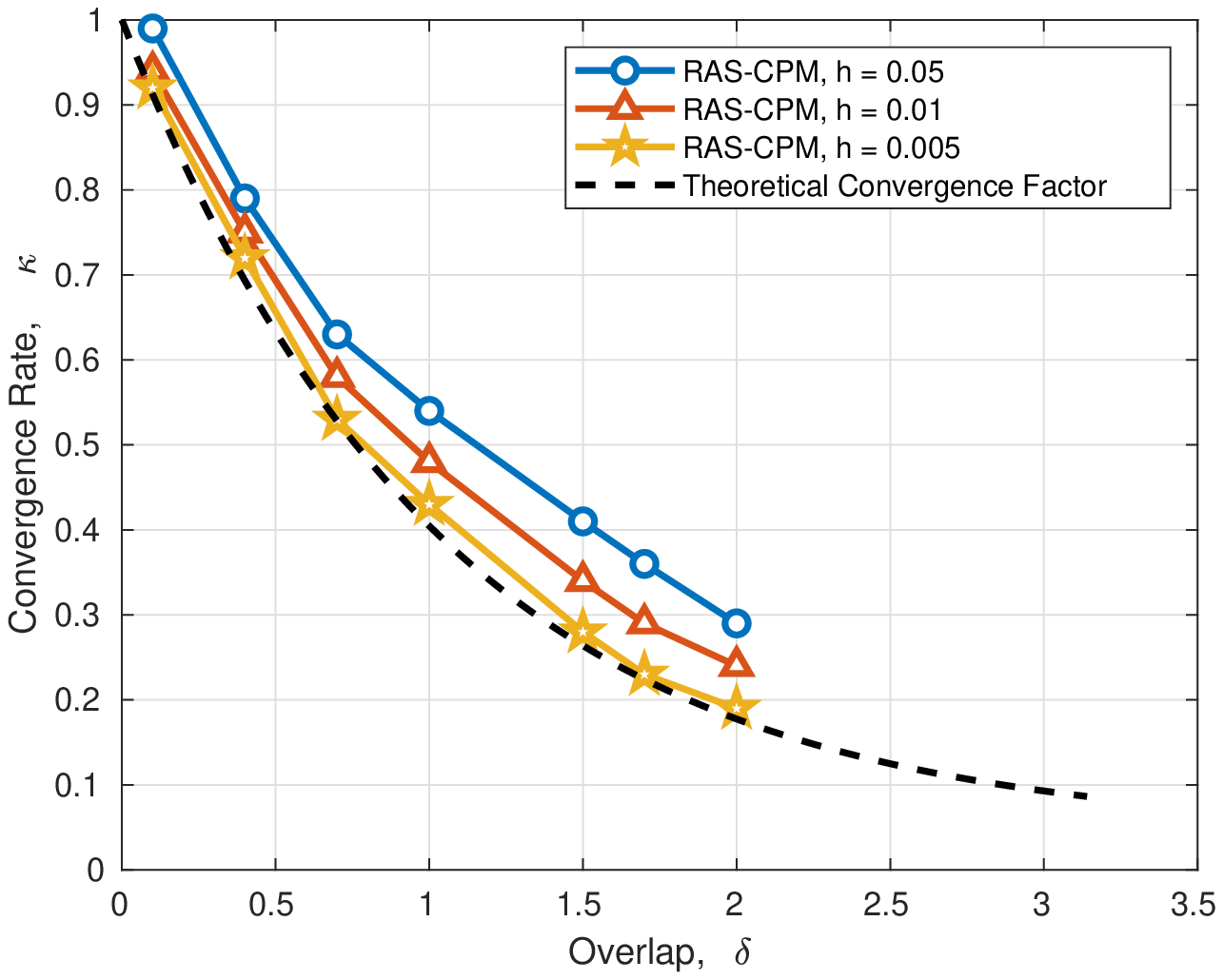}
		\caption{\textbf{Left:} Equal-sized disjoint subdomains for the unit circle. \textbf{Right:} Comparison of the RAS-CPM convergence factor and theoretical convergence factor for different values of overlap parameter in an equal-sized subdomain configuration for the unit circle.}
		\label{fig:yazdani_contrib_Conv}
	\end{figure}
	\section{Conclusion}
	\label{sec:yazdani_contrib_Conc}
	
	Employing RAS as a solver for the CPM parallelizes the solution of PDEs on surfaces and enhances the performance for large scale problems. In this paper, convergence of the (continuous) CPM equipped with a restricted additive Schwarz solver was investigated for a one-dimensional manifold in $\mathbb{R}^d$. Convergence was shown for the two-subdomain case;  extensions to any finite number of subdomains is under investigation \cite{Alireza2021}.   Observed convergence rates agree with our theory as the mesh spacing is refined.   Indeed, the results apply to any convergent discretization (e.g., a finite element discretization) of RAS solvers applied to surface PDEs  as the mesh spacing approaches zero. Finally, note that other variants of Schwarz methods -- sequential restricted additive Schwarz, optimized restricted additive Schwarz, and multiplicative methods -- can be utilized as a solver or a preconditioner for the CPM.  We plan to extend our analysis to these cases as well. 
	\begin{acknowledgement}
		The authors gratefully acknowledge the financial support of NSERC Canada (RGPIN 2016-04361 and RGPIN 2018-04881).
	\end{acknowledgement}
	\bibliographystyle{unsrt}
	\bibliography{yazdani_contrib.bib}
\end{document}